\documentclass[12pt,a4paper]{amsart}
\usepackage{mathrsfs}

\newtheorem{theo+}              {Theorem}           [section]
\newtheorem{prop+}  [theo+]     {Proposition}
\newtheorem{coro+}  [theo+]     {Corollary}
\newtheorem{lemm+}  [theo+]     {Lemma}
\newtheorem{exam+}  [theo+]     {Example}
\newtheorem{rema+}  [theo+]     {Remark}
\newtheorem{defi+}  [theo+]     {Definition}

\newenvironment{theorem}{\begin{theo+}}{\end{theo+}}
\newenvironment{proposition}{\begin{prop+}}{\end{prop+}}
\newenvironment{corollary}{\begin{coro+}}{\end{coro+}}

\newenvironment{definition}{\begin{defi+}}{\end{defi+}}

\usepackage{amsthm}
\theoremstyle{plain} \theoremstyle{remark}
\newtheorem{remark}{Remark}
\newtheorem{example}{Example}

\def \r{\mbox{${\mathbb R}$}}
\def\E{/\kern-1.0em \equiv }

\title{Biharmonic conformal immersions into 3-dimensional manifolds}
\author{Ye-Lin Ou$^{*}$ }

\address{Department of Mathematics,\newline\indent
Texas A $\&$ M University-Commerce,\newline\indent Commerce, TX
75429, U S A.\newline\indent E-mail:yelin.ou@tamuc.edu}

\thanks{$^{*}$ Research supported by
NSF of Guangxi (P. R. China), 2011GXNSFA018127.}
\begin{document}

\title[Biharmonic conformal immersions into 3-dimensional manifolds]
{Biharmonic conformal immersions into 3-dimensional manifolds}

\thanks{}
\date{08/28/2012}

\maketitle
\section*{Abstract}
\begin{quote}
{\footnotesize  Motivated by the beautiful theory and the rich
applications of harmonic conformal immersions and conformal
immersions of constant mean curvature (CMC) surfaces, we study
biharmonic conformal immersions of surfaces into a generic
3-manifold. We first derive an invariant
equation for such immersions, we then try to answer the question,
``what surfaces can be biharmonically conformally immersed into Euclidean 3-space $\r^3$?"  We prove that a circular cylinder is the only CMC surface that can be biharmonically conformally immersed into $\r^3$; We obtain a
classification of biharmonic conformal immersions of complete CMC surfaces
into $\r^3$ and hyperbolic 3-spaces.  We  also study rotational surfaces that can be biharmonically
conformally immersed into $\r^3$ and prove that a circular cone can never be biharmonically
conformally immersed into $\r^3$.}
\end{quote}

\section{Introduction}

In this paper, all manifolds, maps, vector fields are assumed to be smooth and Einstein summation convention is used. \\

A biharmonic map is a map $\varphi:(M, g)\longrightarrow (N, h)$
between Riemannian manifolds that is a critical point of the
bienergy functional
\begin{equation}\nonumber
E_{2}\left(\varphi,\Omega \right)= \frac{1}{2} {\int}_{\Omega}
\left|\tau(\varphi) \right|^{2}{\rm d}x
\end{equation}
for every compact subset $\Omega$ of $M$, where $\tau(\varphi)={\rm
Trace}_{g}\nabla {\rm d} \varphi$ is the tension field of $\varphi$ vanishing of which means $\varphi$ is a harmonic map.
Biharmonic map equation is the Euler-Lagrange equation of this functional which can be written as (\cite{Ji1})
\begin{equation}\notag
\tau^{2}(\varphi):={\rm
Trace}_{g}(\nabla^{\varphi}\nabla^{\varphi}-\nabla^{\varphi}_{\nabla^{M}})\tau(\varphi)
- {\rm Trace}_{g} R^{N}({\rm d}\varphi, \tau(\varphi)){\rm d}\varphi
=0,
\end{equation}
where  $R^{N}$  denotes
the curvature operator of $(N, h)$ with the convention
$$R^{N}(X,Y)Z=
[\nabla^{N}_{X},\nabla^{N}_{Y}]Z-\nabla^{N}_{[X,Y]}Z.$$

{\em Biharmonic submanifolds} are referred to those submanifolds whose defining isometric immersions are giving by biharmonic maps. The notion of biharmonic maps is a natural generalization of that of harmonic maps and biharmonic submanifolds include minimal submanifolds as a subclass. It is well known that harmonic conformal immersions of surfaces  are exactly conformal minimal immersions of surfaces which have been in the focus of study for many decades and the rich theory of which has exhibited a beautiful interplay among geometry, topology and complex analysis (See \cite{CM}, \cite{Co}, \cite{Ke} and the references therein).\\

One motivation of this paper is to explore how far the beautiful theory, useful techniques, and important applications of the minimal surfaces can go in the following direction of generalization:

$$\{ \rm Minimal\; surfaces\; in\; \r^3\}=\{\rm Harmonic\; conformal\; immersions:\; M^2\longrightarrow \r^3\}$$ $$ \subset $$ $$\{\rm Biharmonic\; conformal\; immersions:\; M^2\longrightarrow \r^3\}$$\\

We say a hypersurface in a Riemannian manifold $(N^{m+1}, h)$
defined by an isometric immersion $\varphi: (M^{m},{\bar g})
\longrightarrow (N^{m+1},h)$ can be   biharmonically conformally
immersed into $(N^{m+1}, h)$, if there exists  a function $\lambda
:M^m\longrightarrow \r^{+}$ such that the conformal immersion
$\varphi: (M^{m}, \lambda^{-2}{\bar g}) \longrightarrow (N^{m+1},h)$  is a biharmonic map.\\

In this paper, we study biharmonic conformal immersions of surfaces into a generic 3-dimensional Riemannian manifold. After deriving an invariant
equation for such immersions, we attempt to answer the question, ``what surfaces can be biharmonically conformally immersed into Euclidean 3-space $\r^3$?"  Among other things, we prove that a circular cylinder is the only CMC surface that can be biharmonically conformally immersed into $\r^3$; We obtain a
classification of biharmonic conformal immersions of complete CMC surfaces
into $\r^3$ and hyperbolic 3-spaces.  For non-constant mean curvature surfaces, we  obtain conditions for rotational surfaces that can be biharmonically
conformally immersed into $\r^3$ and prove that a circular cone can never be biharmonically
conformally immersed into $\r^3$.

\section {Conformal biharmonic hypersurfaces}

Biharmonic hypersurfaces in a generic Riemannian manifold were studied in \cite{Ou1} where, among other things,  the following theorem was proved.
\begin{theorem}$($\cite{Ou1}$)$\label{MTH}
Let $\varphi:M^{m}\longrightarrow N^{m+1}$ be an isometric immersion
of codimension-one with mean curvature vector $\eta=H\xi$. Then
$\varphi$ is biharmonic if and only if:
\begin{equation}\notag
\begin{cases}
\Delta H-H |A|^{2}+H{\rm
Ric}^N(\xi,\xi)=0,\\
 2A\,({\rm grad}\,H) +\frac{m}{2} {\rm grad}\, H^2
-2\, H \,({\rm Ric}^N\,(\xi))^{\top}=0,
\end{cases}
\end{equation}
where ${\rm Ric}^N : T_qN\longrightarrow T_qN$ denotes the Ricci
operator of the ambient space defined by $\langle {\rm Ric}^N\, (Z),
W\rangle={\rm Ric}^N (Z, W)$,  $A$ is the shape operator of the
hypersurface with respect to the unit normal vector $\xi$, and $\Delta$ and ${\rm grad}$ denote the Laplace and the gradient operators defined by the induced metric on the hypersurface.
\end{theorem}

\begin{definition}\label{BCI}
A hypersurface in a Riemannian manifold $(N^{m+1}, h)$ defined by an isometric immersion $\varphi: (M^{m},{\bar g}) \longrightarrow (N^{m+1},h)$ is said to admit a {\bf biharmonic conformal immersion} into $(N^{m+1}, h)$, if there exists  a function $\lambda
:M^m\longrightarrow \r^{+}$ such that the conformal immersion
$\varphi: (M^{m}, \lambda^{-2}{\bar g}) \longrightarrow (N^{m+1},h)$ with conformal factor $\lambda$ is a biharmonic map. In such a case, we also say that the hypersurface
 $\varphi: (M^{m},{\bar g}) \longrightarrow (N^{m+1},h)$  can be {\bf biharmonically conformally immersed} into $(N^{m+1}, h)$
\end{definition}
\begin{remark}
(1) Clearly, for every biharmonic conformal immersion $\varphi: (M^{m}, {\bar g}) \longrightarrow (N^{m+1}, h)$ with $\varphi^{*}h=\lambda^2{\bar g}$, the associated hypersurface  $\varphi: (M^{m}, \varphi^{*}h) \longrightarrow (N^{m+1}, h)$ admits biharmonic conformal immersion into $(N^{m+1}, h)$ since $\lambda^{-2}(\varphi^{*}h)= {\bar g}$.\\
(2) It is well known that a surface is minimal if and only if its
defining isometric immersion $\varphi:M^2\longrightarrow (N^n,h)$ is
harmonic. Since harmonicity of a map from a surface is invariant
under  conformal changes  of the metric on the surface we conclude
that  for any positive function $\lambda$, the conformal immersion
$\varphi:(M^2, \lambda^{-2}\varphi^{*}h)\longrightarrow (N^n,h)$ is
again harmonic and hence trivially biharmonic. It follows from
Definition \ref{BCI} that a minimal surface
$\varphi:M^2\longrightarrow (N^n,h)$ can always be {\em trivially}
biharmonically conformally immersed in $(N^n,h)$. For this reason,
in the rest of this paper,  ``biharmonic conformal immersions of
surfaces" will always mean biharmonic conformal immersions of  {\bf
non-minimal}  surfaces.
\end{remark}

\begin{proposition}
A hypersurface $\varphi : (M^{m}, g) \longrightarrow (N^{m+1},h)$
with  mean curvature vector $\eta=H\xi$ with respect to the unit
normal vector field $\xi$ can be biharmonically conformally immersed
into $ (N^{m+1}, h)$  if and only if there exists a function
$\lambda : M \longrightarrow (0, \infty)$ such that
\begin{eqnarray}\notag
\lambda^{4}\tau^{2}(\varphi, g)&=& -(m-2)J^{ \varphi}_{\bar g}({\rm
d}{\varphi}({\rm grad}_{\bar g}\,\ln\lambda)) + 2m\lambda^2(-\Delta_{\bar g} {\rm
ln}\lambda-2\left|{\rm grad}_{\bar g}\,\ln\lambda\right|_{\bar g}^2)\eta\\\label{Confi}
&&+m(m-6)\lambda^2\nabla^{ \varphi}_{{\rm grad}_{\bar g}\,\ln\,\lambda}\,
\eta.
\end{eqnarray}
\end{proposition}
\begin{proof}
Let ${\bar g}=\lambda^{-2}g$, one can easily see that map $\varphi :
(M^{m}, {\bar g}) \longrightarrow (N^{m+1}, h)$ becomes a conformal
immersion since $\varphi^{*}h=\lambda^2 {\bar g}$. The proposition
then follows from Proposition 1 in \cite{Ou2}.
\end{proof}

\begin{theorem}\label{MT10}
A conformal immersion
\begin{equation}\label{MAP1}
\varphi : (M^{2},{\bar g}) \longrightarrow
(N^3,h)
\end{equation}
 into a $3$-dimensional manifold with
$\varphi^{*}h=\lambda^{2}{\bar g}$ is biharmonic if and only if
\begin{equation}\label{M3}
\begin{cases}
\Delta H -H[|A|^2-{\rm Ric}^N(\xi,\xi)-\lambda^{-2}\Delta (\lambda^2)]+4g({\rm grad\;ln}
\lambda,{\rm grad} H)=0,\\A({\rm grad} H)+ H[{\rm grad}
H- \,({\rm Ric}^N\,(\xi))^{\top}+2\,A({\rm grad\;ln}
\lambda)]=0\\
\end{cases}
\end{equation}
where $\xi$, $A$, and $H$ are the unit normal vector field, the
shape operator, and the mean curvature function  of
the surface $\varphi(M)\subset (N^3, h)$ respectively, and the
operators $\Delta,\; {\rm grad}$ and $|,|$ are taken with respect to
the induced metric $g=\varphi^{*}h=\lambda^{2}{\bar g}$ on the
surface.
\end{theorem}
\begin{proof}
Let $\varphi : (M^{2}, \varphi^{*}h=\lambda^2{\bar g})
\longrightarrow (N^3, h)$ be the isometric immersion associated to
the conformal immersion (\ref{MAP1}). Then, the conformal immersion
(\ref{MAP1}) is biharmonic if and only if the associated surface can
be biharmonically conformally immersed into $(N^3, h)$ since
$\lambda^{-2}(\varphi^{*}h)=\bar g$. It follows from (\ref{Confi})
with $m=2$ that the conformal immersion $\varphi$ is biharmonic if
and only if
\begin{eqnarray}\label{GD12}
&&\lambda^{2}\tau^{2}(\varphi,g)=  -4(\Delta_{\bar g} {\rm
ln}\lambda+2{\left|{\rm grad}_{\bar g}\ln\lambda\right|_{\bar
g}}^2)\eta-8\nabla^{ \varphi}_{{\rm grad}_{\bar g}
\ln\,\lambda}\,\eta,
\end{eqnarray}
where $\tau^{2}(\varphi,g)$ denotes the bitension field of the
associated isometric immersion $\varphi :
(M^{2},g=\lambda^{2}\bar{g}) \longrightarrow (N^3, h)$ with mean
curvature vector $\eta=H\xi$, where $\xi$ and $H$ are the unit
normal vector field and the mean curvature function of the surface
$\varphi (M)$ respectively. Using the formula for $\tau^2 (\varphi, g)$ given in \cite{Ou1}  with $m=2$, we
have
\begin{eqnarray}\notag
\tau^{2}(\varphi,g) &=& 2\Big[\Delta H -H\,|A|^2+H{\rm
Ric}^N(\xi,\xi)\Big]\xi \\\notag &-& 2\Big[ 2A({\rm grad} H)+ {\rm
grad} (H^2)-2\, H \,({\rm Ric}\,(\xi))^{\top}\Big].
\end{eqnarray}
Substituting this into (\ref{GD12}) we have
\begin{eqnarray}\label{gd20}
&&\lambda^{2}\big[\Delta H -H\,|A|^2+H{\rm Ric}^N(\xi,\xi)\big]\xi
\\\notag&&- \lambda^{2}[ 2A({\rm grad} H)+ {\rm grad} (H^2)-2\, H \,({\rm Ric}\,(\xi))^{\top}\Big]\\\notag
&=& -2(\Delta_{\bar g} {\rm ln}\lambda+2{\left|{\rm grad}_{\bar
g}\ln\lambda\right|_{\bar g}}^2)\eta-4\nabla^{ \varphi}_{{\rm
grad}_{\bar g} \ln\,\lambda}\,H\xi.
\end{eqnarray}

On the other hand, it is easy to check that the transformations of Laplacian and the gradient
operators under a conformal change of metrics $g=\lambda^{2}\bar{g}$
on a two-dimensional manifold are given by
\begin{eqnarray}\label{La}
\Delta_{\bar{g}} u=\lambda^{2}\Delta u,\;\;\;{\rm
grad}_{\bar{g}}u=\lambda^{2}{\rm grad}u.
\end{eqnarray}
Using these we have
\begin{eqnarray}\label{La1}
-4\nabla^{\varphi}_{{\rm grad}_{\bar{g}}\ln\,\lambda} H\xi=-4{\bar
g}({\rm grad}_{\bar{g}}\ln\lambda,{\rm grad}_{\bar{g}}H) \xi
+4H\,A({\rm grad}_{\bar{g}}\ln\lambda)\\\notag= -4\lambda^2g({\rm
grad}\ln\lambda,{\rm grad}H) \xi +4\lambda^2H\,A({\rm
grad}\ln\lambda)
\end{eqnarray}

Substituting (\ref{La}) and  (\ref{La1}) into (\ref{gd20}) yields
\begin{eqnarray}\notag
&&\big[\Delta H -H\,|A|^2 +H{\rm Ric}^N( \xi,\xi)\big]\xi
\\\notag&&- [ 2A({\rm grad} H)+
{\rm grad} (H^2)-2\, H \,({\rm Ric}\,(\xi))^{\top}]\\\notag &=&
-\lambda^{-2}(\Delta \lambda^2)H\xi-4g({\rm grad\,ln}\lambda,{\rm grad}H)
\xi +4H\,A({\rm grad\;ln}\lambda).
\end{eqnarray}
By comparing the tangential and normal parts of this vector equation we
obtain the theorem.
\end{proof}

\begin{corollary}\label{C2}
A conformal immersion $\varphi : (M^{2}, {\bar
g}) \longrightarrow (N^3(C),
h_0)$ into $3$-dimensional space of constant sectional curvature $C$
with $\varphi^{*}h_{0}=\lambda^{2}{\bar
g}$ is biharmonic if and only if
\begin{equation}\label{C}
\begin{cases}
\Delta H -H[|A|^2-2C-\lambda^{-2}\Delta (\lambda^2)]+4g({\rm
grad\;ln} \lambda,{\rm grad} H)=0,\\A({\rm grad} H)+ H[{\rm
grad} H+2\,A({\rm grad\;ln}
\lambda)]=0,\\
\end{cases}
\end{equation}
where $\xi$ is the unit normal vector field of the surface
$\varphi(M)\subset \mathbb{R}^3$ and $A$ and $H$ are the shape
operator and the mean curvature function of the surface
respectively, and the operators $\Delta,\; {\rm grad}$ and $|,|$ are
taken with respect to the induced metric
$g=\varphi^{*}h=\lambda^{2}{\bar g}$ on the surface.
\end{corollary}
\begin{proof}
This follows from Theorem \ref{MT10} and the fact that a $3$-dimensional space of constant
sectional curvature $C$ is an Einstein manifold with ${\rm
Ric}^N(\xi,\xi)=2C$ and ${\rm Ric}^N\,(\xi))^{\top}=0$.
\end{proof}
\begin{remark}
When $C=0$, Corollary \ref{C2} recovers Theorem 2 in \cite{Ou2}
where the notation $|A|^2$ denotes the norm square of the second
fundamental form taken with respect to the conformal metric ${\bar
g}$ rather than the induced metric.
\end{remark}

Using Theorem \ref{MT10} and Definition \ref{BCI} we have
\begin{corollary}
A surface $\varphi:M^2\longrightarrow (N^3,h)$ with the induced metric $g=\varphi^*h$, the shape operator $A$, and the mean curvature function $H$ can be biharmonically conformally immersed into $ (N^3,h)$  if and only if there exists a positive function $\lambda$ defined on
$M^2$ that solves  Equation (\ref{M3}) .
\end{corollary}

For CMC surfaces, we have
\begin{corollary}
A non-zero constant mean curvature surface
$\varphi:M^2\longrightarrow (N^3,h)$ with the induced metric
$g=\varphi^*h$ and the shape operator $A$ can be harmonically
conformally immersed into $ (N^3,h)$  if and only if and only if
there exists a positive function $\lambda$ defined on $M^2$ such
that
\begin{equation}\label{ContH}
\begin{cases}
\Delta (\lambda^2)=\lambda^{2}[|A|^2-{\rm Ric}^N(\xi,\xi)],\\A({\rm
grad\;ln} \lambda)= \frac{1}{2}({\rm Ric}^N\,(\xi))^{\top}
\end{cases}
\end{equation}
\end{corollary}

It is well known (see e.g., \cite{CI}, \cite{Ji2}) that if a surface admits a biharmonic homothetic immersion into Euclidean 3-space $\r^3$, then it has to be a minimal surface. It was proved in \cite{Ou2} that a circular cylinder in $\r^3$ can be biharmonically conformally immersed into $\r^3$. As to the general question: what surfaces in $\r^3$ can be biharmonically conformally immersed into $\r^3$,  we have

\begin{corollary}
A surface $ M^2\longrightarrow\r^3$ with mean curvature function $H$ and the
shape operator $A$ can be biharmonically conformally immersed into
$\r^3$ if and only if there exists a positive function $\lambda$ defined on
$M^2$ such that
\begin{equation}\label{M310}
\begin{cases}
\Delta H -H[|A|^2-\lambda^{-2}\Delta (\lambda^2)]+4g({\rm grad\;ln}
\lambda,{\rm grad} H)=0,\\A({\rm grad} H)+ H[{\rm grad}
H+2\,A({\rm grad\;ln}
\lambda)]=0,\\
\end{cases}
\end{equation}
where  the Laplace operator $\Delta,\; {\rm grad}$ and $|,|$ are
taken with respect to
the induced metric $g=\varphi^{*}h_0$ on $M$.\\

In particular, a  surface $M^2$ in $\r^3$ of constant mean curvature $H\ne 0$ can be biharmonically conformally immersed into $\r^3$ if and only if there exists a positive function $\lambda$ defined on
$M^2$ such that
\begin{equation}\label{R3}
\begin{cases}
-|A|^2+ \lambda^{-2}\Delta\lambda^2=0,\\A({\rm grad\;ln}
\lambda)=0.\\
\end{cases}
\end{equation}
\end{corollary}

By the nature of the PDE (\ref{M310}),  the question what surfaces can be biharmonically conformally immersed into $\r^3$ does not seem to be a simple one. Even in the family of the CMC surfaces, we know  that a plane (being totally geodesic and hence minimal) always admits trivial biharmonic conformal immersions, we also know (\cite{Ou2}) that  a circular cylinder can be biharmonically conformally immersed into $\r^3$. However,  as the following corollary shows,  not all CMC surfaces enjoy this property.

\begin{corollary}\label{S2}
No part of the standard sphere $S^2$ can be biharmonically
conformally immersed into $\mathbb{R}^3$.
\end{corollary}
\begin{proof}
One can parametrize a piece of the unit sphere as $\phi: D\longrightarrow \mathbb{R}^3$, with $ \phi(u,v)=(\cos u \cos v, \cos u\sin v, \sin u)$, then, a straightforward computation yields the shape operator as
\begin{equation}\label{R4}
A= \left(\begin{array}{cc} 1 & 0  \\
0  & \cos^2 u
\end{array}\right).
\end{equation}
If there were a biharmonic conformal immersion of the unit sphere into $\r^3$. Then, Equation (\ref{R4}),  together with the second equation of (\ref{R3}), would imply that ${\rm grad\;ln}
\lambda=0$, and hence $\lambda={\rm constant}$. It follows from this and the first equation of  (\ref{R3}) that  $|A|^2=0$, which is clearly a contradiction to (\ref{R4}). Thus, we obtain the corollary.
\end{proof}
Our next theorem shows that circular cylinders are the only constant mean curvature surfaces that admit biharmonic conformal immersions into $\r^3$. To prove the theorem, we will need  the following theorem  which gives the existence of special isothermal coordinates on a nonzero constant mean curvature surface and the evidence of the existence of many CMC surfaces in $\r^3$.

{\bf Theorem A}  (see, e.g., \cite{Ke}, p.22):  Let
$M^2\longrightarrow\r^3$ be a surface of nonzero constant mean
curvature $H$, and let $p\in M$ be a non-umbilical point. Then there
exist isothermal coordinates $(u, v)$ in a neighborhood of $p$
satisfying
\begin{eqnarray}\label{CI}
I&=&\frac{e^{2w}}{2H}(du^2+dv^2),\\\label{CII}
II&=&e^w\cosh w\, du^2+e^w\sinh w\, dv^2,
\end{eqnarray}
where w=w(u,v) is a solution of the sinh-Gordon equation
\begin{equation}\notag
w_{uu}+w_{vv}+2H\cosh w \sinh w=0.
\end{equation}
Conversely, for any given positive constant $H$ and a solution $w$ of sinh-Gordon equation, there exists a CMC surface whose first and the second fundamental forms are given by (\ref{CI}), and (\ref{CII}).\\

Now, we are ready to prove the following theorem.
\begin{theorem}
A constant mean curvature surface can be biharmonically conformally immersed into $\r^3$ if and only if it is a part of a plane or a circular cylinder.
\end{theorem}
\begin{proof}
Let $\varphi: (M^2, \varphi^{*}h_0)\longrightarrow(\r^3, h_0)$ be a surface of constant mean curvature $H$. If the surface is totally umbilical, then it is well known that it is a part of a plane or a sphere. As we mentioned in the paragraph preceding Corollary \ref{S2} that a plane can always be (trivially) biharmonically conformally immersed in $\r^3$. On the other hand, we know from Corollary \ref{S2} that no part of a sphere can be biharmonically conformally immersed into $\r^3$.\\

If the surface is not totally umbilical, then, by Theorem A, we can choose local isothermal coordinates $(u,v)$ so that its first and second fundamental forms are given by (\ref{CI}) and (\ref{CII}). Thus, the coefficients of the first and the second fundamental forms are given by
\begin{eqnarray}\notag
g_{11}=  g_{22} =\frac{e^{2w}}{2H}, \;\;g_{12}=g_{21}=0;\;\;\;g^{11}=g^{22}=\frac{2H}{e^{2w}},\;\; g^{12}=g^{21}=0.
\end{eqnarray}
and
\begin{eqnarray}\notag
h_{11}=e^w \cosh w,\;\;\; h_{12}=h_{21}= 0, \;\;\;h_{22}=e^w\sinh w.
\end{eqnarray}
A straightforward computation yields
\begin{eqnarray}\notag
A(\partial_1)&=&g^{kl}h_{l1}\partial_k=2He^{-w}\cosh w \,\partial_1,\\\notag
A(\partial_2)&=&g^{kl}h_{l2}\partial_k=2He^{-w}\sinh w\, \partial_2,\\\notag
{\rm grad}({\rm ln} \lambda)&=&2He^{-2w}[(\partial_1 \ln \lambda)\partial_1+(\partial_2 \ln \lambda)\partial_2],\\\label{GD30}
A({\rm grad}({\rm ln} \lambda))&=&4H^2e^{-3w}[(\partial_1 \ln \lambda)\cosh w \,\partial_1+(\partial_2 \ln \lambda)\sinh w\,\partial_2].
\end{eqnarray}
Substituting (\ref{GD30}) into the second equation of (\ref{R3}) we
obtain
\begin{equation}\label{GD31}
\begin{cases}
4H^2e^{-3w}(\partial_1 \ln \lambda)\cosh w =0,\\
4H^2e^{-3w}(\partial_2 \ln \lambda)\sinh w=0.
\end{cases}
\end{equation}
Since $H\ne 0$ we solve (\ref{GD31})  to have $\lambda={\rm
constant}$, or $w=0$ and $\partial_1 \ln \lambda=0$. In the case of
$\lambda={\rm constant}$, the biharmonic conformal immersion
$\varphi: (M^2, \lambda^{-2}\varphi^{*}h_0)\longrightarrow(\r^3,
h_0)$ is actually a biharmonic homothetic immersion which is minimal
in $\r^3$ by a well known result in \cite{CI}. This contradicts to
the assumption that $H\ne 0$. In the case of $w=0$ and $\partial_1
\ln \lambda=0$,  it is not difficult to see from (\ref{CI}) and
(\ref{CII}) that the CMC surface is isometric to a circular
cylinder. Conversely, we know that a plane always admits trivial
biharmonic conformal immersion into $\r^3$. We also know from
\cite{Ou2} that a circular cylinder always admits a biharmonic
conformal immersion into $\r^3$. Summarizing the above results we
obtained the theorem.
\end{proof}

Our next theorem gives conditions for a rotational surface that can be biharmonically conformally immersed into $\r^3$.

\begin{theorem}\label{RS}
A non-minimal rotational surface $r(u, v)=(f(u)\cos v,\; f(u) \sin v,\; g(u))$ obtained by rotating the arclength parametrized curve $ (C): x=f(u), z=g(u)$ about $z$-axis can be biharmonically conformally immersed into $\r^3$ if and only if there exists a positive function $\lambda $ on the surface depending only on variable $u$ such that
\begin{eqnarray}\label{ROT}
&& k H'+HH'+2k H (\ln \lambda)'=0,\\
&& 2(\ln \lambda)''+2[\ln (fH^2)]'(\ln \lambda)'+4(\ln \lambda)'^2=\\\notag
&& k^2+(g'f^{-1})^2-H''H^{-1}-f'f^{-1}H'H^{-1},
\end{eqnarray}
where $k=-f''g'+g''f'$ is the curvature of the generating curve (C),
$H=[k+g'f^{-1}]/2$ is the mean curvature of the rotational surface.
\end{theorem}
\begin{proof}
A straightforward computation gives
\begin{eqnarray}\notag
r_u&=&(f'\cos v,\; f'\sin v, \;g'),\;\;\;r_v=(-f\sin v,\; f\cos v, \;0),\\\notag r_{uu}&=&(f''\cos v,\; f''\sin v, \;g''),\;\;\;\;r_{uv}=(-f'\sin v,\; f'\cos v, \;0),\\\notag
r_{vv}&=&(-f(u)\cos v,\; -f(u) \sin v,\; 0).
\end{eqnarray}
The unit normal vector field of the rotational surface  can be chosen to be $\xi=\frac{r_u\times r_v}{| r_u\times r_v|}=(-g'\cos v, -g'\sin v, f')$. A further computation gives the coefficients of the first fundamental form
\begin{eqnarray}\notag
g_{11}= \langle r_{u},  r_u\rangle =1, \;\;g_{12}=g_{21}=\langle r_{u},  r_v \rangle=0,\;\; g_{22}=\langle  r_v , r_v\rangle=f^2, \\\notag{\rm and\;hence}\;\;\;g^{11}=1,\;\; g^{12}=g^{21}=0,\;\; g^{22}=f^{-2}
\end{eqnarray}
  and the second fundamental form
\begin{eqnarray}\notag
h_{11}=\langle r_{uu}, \xi\rangle=-f''g'+g''f'=k,\;\;\; h_{12}=h_{21}= \langle r_{uv}, \xi\rangle=0, \;\;\;h_{22}=\langle r_{vv}, \xi\rangle=fg'.
\end{eqnarray}

Using the natural frame $\{\partial_1=r_u, \partial_2=r_v, \xi\}$ on $\r^3$ adapted to the rotational surface  we compute
\begin{eqnarray}\notag
A(\partial_1)&=&g^{ij}h_{j1}\partial_i=k \partial_1,\\\notag
A(\partial_2)&=&g^{ij}h_{j2}\partial_i=g'f^{-1}\partial_2,\\\notag
|A|^2&=&g^{kl}g^{ij}h_{ki}h_{jl}=k^{2}+(g'f^{-1})^2,\\\notag
H&=&\frac{1}{2}g^{ij}h_{ij}=(k+g'f^{-1})/2,\\\label{gH}
{\rm grad} H&=& g^{kl}\partial_k(H)\partial_l=H'\partial_1,\\\label{AgH}
A({\rm grad} H)&=&kH'\,\partial_1,\\\notag
{\rm grad}({\rm ln} \lambda)&=&(\partial_1 \ln \lambda)\partial_1+f^{-2}(\partial_2 \ln \lambda)\partial_2,\\\label{Agln}
A({\rm grad}({\rm ln} \lambda))&=&k(\partial_1 \ln \lambda)\partial_1+g'f^{-3}(\partial_2 \ln \lambda) \partial_2.
\end{eqnarray}
Substituting (\ref{gH}), (\ref{AgH}), and (\ref{Agln}) into the second equation of (\ref{M310}) we obtain
\begin{equation}\label{GD40}
\begin{cases}
kH'+HH'+2kH(\partial_1 \ln \lambda)=0,\\
2Hg'f^{-3}(\partial_2 \ln \lambda) =0.
\end{cases}
\end{equation}

Since the rotational surface is assumed to be non-minimal, we have $H\ne 0$. This, together with the second equation of (\ref{GD40}) implies that
\begin{equation}\notag
\partial_2 \ln \lambda=0,
\end{equation}
which means $\lambda$ depends only on the variable $u$.\\

In order to expand the first equation of (\ref{M310}) we compute
\begin{eqnarray}\label{DE1}
\Delta H&=& g^{ij}H_{ij}-g^{ij}\Gamma_{ij}^kH_k=H''-(\Gamma^1_{11}+\Gamma^1_{22})H'=H''+f'f^{-1}H',
\end{eqnarray}
where in obtaining the second equality we have used the fact that $H$ depends only on variable $u$. A similar computation yields
\begin{eqnarray}\label{DE2}
\lambda^{-2}\Delta\lambda^2 &=& 2\Delta(\ln \lambda) + 4|\rm grad \ln \lambda|^2\\\notag &=&2 (\ln \lambda)''+2f'f^{-1}(\ln \lambda)'+4(\ln \lambda)'^2.
\end{eqnarray}
Substituting (\ref{DE1}) and (\ref{DE2}) into the first equation of (\ref{M310}) and simplifying the result we obtain
\begin{eqnarray}\label{DE3}
&& 2(\ln \lambda)''+2[\ln (fH^2)]'(\ln \lambda)'+4(\ln \lambda)'^2=\\\notag
&& k^2+(g'f^{-1})^2-H''H^{-1}-f'f^{-1}H'H^{-1}.
\end{eqnarray}
Combining (\ref{GD40}) and (\ref{DE3}) we obtain the theorem.
\end{proof}

As an application of Theorem \ref{RS} we have
\begin{corollary}
A circular cone can never be biharmonically conformally immersed into $\r^3$.
\end{corollary}
\begin{proof}
A circular cone is a rotational surface with the generating curve
being a straight line whose curvature $k=0$. It follows from
Equation of (\ref{ROT}) that if a circular cone admitted a
biharmonic conformal immersion into $\r^3$, then $H'=0$ and hence
its mean curvature  would be constant, which is clearly a
contradiction since the mean curvature of a circular cone is not
constant.
\end{proof}

Our next theorem gives a classification of biharmonic conformal immersions of complete constant mean curvature surfaces into $\r^3$.

\begin{theorem}\label{MT2}
Let $\varphi : (M^{2},{\bar g}) \longrightarrow (\r^3, h_0)$ be a
biharmonic conformal immersion of a surface into
$3$-dimensional Euclidean space with
$\varphi^{*}h_0=\lambda^{2}{\bar g}$ being complete and $\int_M\lambda^6 \,dv_{\bar
g}<\infty$.  If the surface $\varphi(M)\subset \r^3$ has constant
mean curvature, then the biharmonic conformal immersion $\varphi$ is
minimal.
\end{theorem}

\begin{proof}
By Corollary \ref{C2}, the equation for a biharmonic conformal immersion into $\r^3$ reduces to
\begin{equation}\label{M31}
\begin{cases}
\Delta H -H[|A|^2-\lambda^{-2}\Delta (\lambda^2)]+4g({\rm grad\;ln}
\lambda,{\rm grad} H)=0,\\A({\rm grad} H)+ H[{\rm grad}
H+2\,A({\rm grad\;ln}
\lambda)]=0\\
\end{cases}
\end{equation}
where  $A$ and $H$ are  the shape operator, and the mean curvature
function of the surface  $\varphi(M)\subset (\r^3, h_0)$
respectively, and the operators $\Delta,\; {\rm grad}$ and $|,|$ are
taken with respect to
the induced metric $g=\varphi^{*}h0=\lambda^{2}{\bar g}$ on $M$.\\

If the surface  $\varphi(M)\subset (\r^3, h_0)$ has constant mean
curvature $H=0$, then nothing is left to  prove. Otherwise, if
$H={\rm constant}\ne 0$, then Equation (\ref{M31}) reads
\begin{equation}\label{M32}
\begin{cases}
\Delta \lambda^2=|A|^2 \lambda^2,\\
A({\rm grad\;ln}
\lambda)=0.
\end{cases}
\end{equation}

Noting that the condition $\int_M\lambda^6 \,dv_{\bar g}<\infty$ and
the first equation in (\ref{M32}) means that $\lambda^2$ is an $L^2$
solution of the Schr$\ddot{\rm o}$dinger type equation on a complete
Riemannian manifold $(M^2, g)$ we use Lemma 3.1 in \cite{NU} to
conclude that $\lambda^2$ is a constant. Thus, the biharmonic
conformal immersion $\varphi$ is a actually a homothetic biharmonic
immersion into $\r^3$ which has to be a minimal immersion and hence
$H=0$, a contradiction. This completes the proof of the theorem.
\end{proof}

\begin{remark}
(1) We would like to point out that a similar argument can be used
to prove that Theorem \ref{MT2} remains true if the target space
$\r^3$
is replaced by any $3$-manifold with non-positive sectional curvature.\\
(2)  We also point out that as the following example shows the
condition $\int_M\lambda^6 dv_{\bar g}<\infty$ in Theorem \ref{MT2}
is sharp and cannot be dropped.
\end{remark}
\begin{example}
Let $\r^2$ be the Euclidean plane provided with the metric $\bar{
g}=e^{-y/R}(dx^2+dy^2$. The map $\varphi:(\r^2, {\bar
g})\longrightarrow \r^3$ given by $\varphi(x,y)=(R\cos\,\frac{x}{R},
R\sin\,\frac{x}{R}, y)$ is a biharmonic conformal immersion of
$\r^2$ into Euclidean space $\r^3$ with $\varphi^{*}h_0=e^{y/R}\bar{
g}$. The induced metric $g=e^{y/R} \bar{ g}$ on $\r^2$ is complete
and the surface $\varphi(\r^2)\subset \r^3$ is a circular cylinder
which has constant mean curvature. The biharmonic conformal
immersion $\varphi$  is not harmonic because $\int_M\lambda^6
dv_{\bar g}=\infty$.
\end{example}
To understand the example, we notice that in this case, $\varphi_x=(-\sin \frac{x}{R}, \cos
\frac{x}{R}, 0)$ and $\varphi_y=(0, 0, 1)$ so one can check that $\varphi$ is a
conformal immersion with $\varphi^{*}h_{0}=dx^2+dy^2=\lambda^2\bar{ g}$ for $\lambda^2=e^{y/R}$. It follows that the induced metric $g=\varphi^{*}h_{0}=dx^2+dy^2$ on $\r^2$ is the standard Euclidean metric which is
complete. One can easily check that $e_1=\varphi_x=(-\sin
\frac{x}{R}, \cos \frac{x}{R}, 0),\;\;e_2=\varphi_y=(0, 0, 1),
\;\;\; \xi=(\cos \frac{x}{R}, \sin \frac{x}{R}, 0)$ form an
orthonormal frame adapted to the surface. A straightforward
computation yields

\begin{eqnarray}\notag
Ae_1&=&-\frac{1}{R}e_1,\;\;Ae_2=0,\\\notag
H&=&\frac{1}{2}(\langle Ae_1,e_1\rangle+\langle
Ae_2,e_2\rangle)=-\frac{1}{2R}\ne 0\\\notag
|A|^2&=&\sum_{i=1}^2|Ae_i|^2=\frac{1}{R^2},\\\notag
{\rm grad} (\ln \lambda)&=&e_2(\ln \lambda)e_2.
\end{eqnarray}
It follows that
\begin{equation}\notag
\begin{cases}
\Delta (\lambda^2)=\Delta (e^{y/R})=\frac{\partial^2}{\partial y^2}(e^{y/R})=\frac{1}{R^2}(e^{y/R})=|A|^2\lambda^2,\\
A({\rm grad\;ln}
\lambda)=0,
\end{cases}
\end{equation}
which means Equation (\ref{M32}) holds. Thus, the conformal
immersion $\varphi$ is indeed a biharmonic map which is not harmonic
since $H\ne 0$. This does not contradict the conclusion of  Theorem
\ref{MT2} because for this example the condition that
$\int_M\lambda^6 dv_{\bar g}<\infty$ required by the theorem is not
satisfied. In fact, $\int_M\lambda^6 dv_{\bar
g}=\infty$.\\

In a very recent paper \cite{NUG}, Nakauchi, Urakawa and Gudmundsson prove that any biharmonic map with finite energy and finite bi-energy from a complete Riemannian manifold into a Riemannian manifold of non-positive sectional curvature  has to be harmonic. We can prove the following results about biharmonic conformal immersions which are the dual
 results  for biharmonic horizontally conformal submersions (Theorem 4.2 and Corollary 4.3) given in \cite{NUG}.
\begin{proposition}\label{PU}
Any biharmonic conformal immersion $\phi: (M^m, g)\longrightarrow (N^n, h)$ from a complete manifold into a nonpositively curved space with $\phi^*h=\lambda^2$ satisfying
\begin{eqnarray}\notag
\int_M\lambda^2dv_g<\infty,\; {\rm and}\\\notag
\int_M|m\lambda^2\eta+(2-m)d \phi({\rm grad \; ln }\lambda)|^2dv_g<\infty,
\end{eqnarray}
where $\eta$ denotes the mean curvature vector of the submanifold $\phi (M)\subset (N, h)$,
has to be harmonic. In particular, any biharmonic conformal immersion $\phi: (M^2, g)\longrightarrow (N^n, h)$ from a complete surface into a nonpositively curved space with $\phi^*h=\lambda^2$ satisfying
\begin{eqnarray}\label{EQ20}
\int_M\lambda^2dv_g<\infty,\; {\rm and}\\\label{EQ21}
\int_M\lambda^4|\eta|^2dv_g<\infty,
\end{eqnarray}
where $\eta$ denotes the mean curvature vector of the submanifold $\phi (M)\subset (N, h)$,
has to be minimal.
\end{proposition}
\begin{proof}
For conformal immersion $\phi: (M^m, g)\longrightarrow (N^n, h)$
with $\phi^*h=\lambda^2 g$, one can easily compute its energy to
have
\begin{equation}\label{En1}
E(\phi)=\int_M|d \phi|^2dv_g=\int_Mg^{ij}\phi^{\alpha}_i\phi^{\beta}_jh_{\alpha\beta}dv_g=m\int_M\lambda^2dv_g.
\end{equation}
On the other hand, a straightforward computation one obtains the
tension field of the conformal immersion $\tau(\phi)=
m\lambda^2\eta+(2-m)d \phi({\rm grad \; ln }\lambda)$. It follows
that the bienergy of the conformal immersion is given by
\begin{eqnarray}\label{En2}
E_2(\phi)=\int_M|m\lambda^2\eta+(2-m)d \phi({\rm grad \; ln }\lambda)|^2dv_g.
\end{eqnarray}
Using Equations (\ref{En1}) and (\ref{En2}) and  Theorem 2.3 in \cite{NUG} we obtain the proposition.
\end{proof}
\begin{remark}
It is interesting to note that we can use Proposition \ref{PU} to
conclude that a biharmonic conformal immersion $\phi: (M^2, {\bar
g})\longrightarrow \r^3$  of a complete CMC surface into Euclidean
3-space with $\phi^*h=\lambda^2{\bar g}$ satisfying
\begin{eqnarray}\label{EQ20}
\int_M\lambda^2dv_{\bar g}&<&\infty,\; {\rm and}\\\label{EQ21}
\int_M\lambda^4 dv_{\bar g}&<&\infty
\end{eqnarray}
has to be minimal. On the other hand, we  can use Theorem \ref{MT2} to have the same conclusion with the assumption that $\int_M\lambda^6 dv_{\bar g}<\infty$. We also know that in general conditions (\ref{EQ20}) and  (\ref{EQ21}) do not imply $\int_M\lambda^6 dv_{\bar g}<\infty$ as $L^2$ space is not closed under multiplication.
\end{remark}

To finish the paper, we give a result that shows the vertical cylinders in $S^2\times \r$ admit biharmonic conformal immersions into $S^2\times \r$, a 3-manifold of nonconstant sectional curvature.
\begin{proposition}
Let $\alpha: I\longrightarrow (S^2, h)$ be (a part of) a circle in
$S^2$ with radius $1/k,\;\; (k>1)$ and let $\Sigma =\bigcup_{t\in
I}\pi^{-1}(\alpha(t))$ denote the vertical cylinder in $S^2\times
\r$. Then, the conformal immersion $\varphi:(\Sigma, {\bar
g}=\lambda^{-2}\varphi^{*}h)\longrightarrow (N^3,h)$ is biharmonic
if and only if $\lambda^2=\big(C_2e^{\pm z/R}-C_1C_2^{-1}R^2e^{\mp
z/R}\big)/2$, where $R=1/\sqrt{k^2-1}$.
\end{proposition}
\begin{proof}
We know as in \cite{Ou1} that the cylinder has constant mean
curvature $H=k/2$, and an adapted orthonormal frame $\{ X, V, \xi\}$
with $\xi$ being normal to the cylinder. The shape operator  and the second fundamental form $b$ with
respect to the orthonormal frame are given by (see \cite{Ou1})
\begin{eqnarray}\notag
&& A(X)=-\langle \nabla_X \xi, X\rangle X-\langle \nabla_X \xi,
V\rangle V=k X,\\\notag && A(V)=-\langle \nabla_V \xi, X\rangle
X-\langle \nabla_V \xi, V\rangle V=-\tau X =0;\\\notag &&
b(X,X)=\langle A(X), X\rangle=k,\;\;b(X,V)=\langle A(X),
V\rangle=-\tau=0,\\\notag && b(V,X)=\langle A(V),
X\rangle=-\tau=0,\;\;b(V,V)=\langle A(V), V\rangle=0.
\end{eqnarray}
A further computation gives
\begin{eqnarray}\notag
&&H=\frac{1}{2}(b(X,X)+b(V,V))=\frac{k}{2},\\\notag && A({\rm
grad}\,H)=A(X(\frac{k}{2})X+V(\frac{k}{2})V)=X(\frac{k}{2})A(X)=\frac{k'}{2}(k
X -\tau V)=0;\\\notag && \Delta H=XX(H)-(\nabla_XX)
H+VV(H)-(\nabla_VV) H=\frac{k''}{2}=0;\\\notag &&
|A|^2=(b(X,X))^2+(b(X,V))^2+( b(V,X))^2+(b(V,V))^2=k^2.
\end{eqnarray}
We also have \cite{Ou1}
\begin{equation}\notag
\begin{cases}
{\rm Ric}(\xi,\xi)={\rm
Ric}(y'E_1-x'E_2,y'E_1-x'E_2)=y'^2+x'^2=1,\\{\rm Ric}(\xi,X)={\rm
Ric}(y'E_1-x'E_2,x'E_1+y'E_2)=x'y'-y'x'=0,\\{\rm Ric}(\xi,V)={\rm
Ric}(y'E_1-x'E_2,E_3)=0.
\end{cases}
\end{equation}
Substituting these into (\ref{ContH}) we have
\begin{equation}\label{CD20}
\begin{cases}
-(k^2-1)+2[\Delta {\rm ln}\lambda+2\left|{\rm
grad\,ln}\lambda\right|^2]=0,\\A({\rm grad\;ln}
\lambda)=0.\\
\end{cases}
\end{equation}
Noting that the cylinder can be parametrized by
$\phi(s,z)=(\alpha(s),z)\subset S^2\times \r$ and that $A({\rm
grad\;ln} \lambda)=X(\ln \lambda)=0$, where $X$ is tangent to
$\alpha$, we conclude that $\lambda(s, z)$ depends only on $z$. It
follows that Equation (\ref{CD20}) is equivalent to
\begin{equation}\notag
\begin{cases}
({\rm ln}\lambda)''+2({\rm ln}\lambda)'^2=(k^2-1)/2,\\
\lambda(s, z)=\lambda(z)
\end{cases}
\end{equation}
or,
\begin{equation}\notag
(\lambda^2)''=\frac{1}{R^2}\lambda^2,
\end{equation}
where, $R=1/\sqrt{k^2-1}$. It follows that $\lambda^2$ is a solution
of the ordinary differential equation
\begin{equation}\notag
y''=\frac{1}{R^2}y,
\end{equation}
which has (see e.g., \cite{Cu}) the first integral
\begin{equation}\label{GD10}
y'^2=y^2/R^2+C_1.
\end{equation}
Solving Equation (\ref{GD10}) we have
\begin{equation}\notag
y=\big(C_2e^{\pm z/R}-C_1C_2^{-1}R^2e^{\mp z/R}\big)/2.
\end{equation}
From this we have
\begin{equation}\notag
\lambda^2=\big(C_2e^{\pm \sqrt{k^2-1}\;z}-C_1[C_2(k^2-1)]^{-1}e^{\mp
\sqrt{k^2-1}\;z}\big)/2.
\end{equation}
Thus, we obtain the proposition.
\end{proof}

\end{document}